\documentclass[titlepage,draft,11pt]{article} 
\usepackage{amsfonts,latexsym} 
\usepackage{pstricks,pst-plot}
\usepackage{amsmath} 
\usepackage{amsthm}
\usepackage[a4paper]{geometry}

\date{}

\newcommand{\re}{\mathbb{R}}

\newtheorem{thm}{Theorem}[section]

\newtheorem{rmk}[thm]{Remark}

\newtheorem{cor}[thm]{Corollary}

\def\liminf{\mathop{\underline{\rm lim}}}

\title{On a linearly damped 2 body problem }

\author{Alain Haraux\vspace{1ex}\\ 
{\normalsize Sorbonne Universit\'e, Universit\'e Paris-Diderot SPC, CNRS, INRIA}, \\
{\normalsize Laboratoire Jacques-Louis Lions,  LJLL, F-75005,
Paris, France.}\\ 
{\normalsize e-mail: \texttt{haraux@ann.jussieu.fr}}}

\begin{document}
\maketitle
\begin{abstract}
The usual equation for both motions of a single planet around the sun and electrons  in the deterministic  Rutherford-Bohr  atomic model is conservative with a singular potential at the origin. When a dissipation is added, new phenomena appear. It is shown that whenever the momentum is not zero, the moving particle does not  reach the center in finite time and its displacement does not blow-up either, even in the classical context where arbitrarily large velocities are allowed.  Moreover we prove that all bounded solutions tend to $0$ for $t$ large, and  some formal calculations suggest the existence of special orbits with an asymptotically spiraling exponentially fast convergence to the center.  \\

\vspace{1cm} 

\end{abstract}

 
\section{Introduction}  The usual equation for both motions of a single planet around the sun and electrons  in the deterministic  Rutherford-Bohr  atomic model is conservative with a singular potential at the origin. In the previous paper \cite{Tired}, the author raised the hypothesis that some phenomena such as carboniferous gigantism of arthropods and vegetals (cf.\cite {HKV, Parks} ), the appearent expansion of the universe  (cf.\cite{Hubble})or slight deformation of electrical devices over decades might be explained by an atomic contraction phenomenon (with respect to the size of the proton) coming from a very weak environmental dissipation produced by what we consider to be vacuum.  When a dissipation is added to the basic equation $$ mu'' = - \frac{q^2}{4\pi \varepsilon_0} \frac {u} {|u|^3} $$  modelling Coulomb's central force (with $q$ the elementary charge, m the mass of the electron  and $\varepsilon_0$  the vacuum permittivity) or its  equivalent  Newton's law for planets (where $G$ is the gravitational constant and $M_S$ the mass of the sun) $$ u'' = - GM_S \frac {u} {|u|^3} $$ written in complex form in the orbital plane with a suitable set of axes and length unit,  new interesting phenomena appear. This leads us to consider,  as the simplest possible dissipative perturbation of the conservative problem \begin{equation} \label{cons} u'' + c_{0}  \frac {u} {|u|^3} = 0 \end{equation}  including both equations,  the modified ODE \begin{equation}\label{diss} u'' + \delta u' + c_{0}  \frac {u} {|u|^3} = 0 \end{equation}In \cite{Tired}, considering the difficulty of exhibiting decaying solutions for the equation \eqref{diss}, the author investigated a slightly more complicated case where we allow the charge $q$ to decay exponentially in time. This led to the equation
  \begin{equation} \label{tired} u'' + \delta u' + c_{0} e^{-\alpha t} \frac {u} {|u|^3} = 0 \end{equation}
   for which explicit exponentially decaying solutions depending on two real parameters can be exhibited. More precisely for $\alpha = \delta$ we found the fast rotating spiraling solutions 
   $$ u_\pm(t) = U e^{-\delta t}  \exp ( \pm i \frac{ \sqrt{c_0}}{\delta U^{3/2}} e^{\delta t} + i \phi) $$ whereas for $\alpha = \frac{3}{2}\delta$ , we found the uniformly rotating spiraling solutions  $$ v_\pm(t) =V e^{-\frac{\delta}{2} t}  \exp ( \pm i t \sqrt{V^{-3}c_0 +  \frac{\delta^2}{4} }  + i \phi) $$
Although it would probably be interesting to do, at the level of \cite{Tired} we did not try to prove well-posedness (global existence) for general initial data not leading to collision with the center, nor did we try to elucidate global behavior of trajectories other than the previous special  solutions. Before deciding to do that, one must first try to see whether or not the simpler model \eqref{diss} is sufficient to identify a contraction phenomenon, and possibly determine the extent of its stability. \\

In the present paper, we derive some partial results on \eqref{diss} together with some heuristic discussion. More precisely, in Section 2, we prove that whenever the momentum is not zero, the moving particle does not  reach the center in finite time and its displacement does not blow-up in finite time either, even in the classical context where arbitrarily large velocities are allowed. In Section 3, we prove that  $u(t)$ remains bounded under a smallness condition involving both initial radius and velocity. In Section 4, we prove that all bounded non-vanishing solutions converge to $0$ at infinity in $t$. Finally, in Section 5, some formal calculations suggest the existence of special orbits with an asymptotically spiraling convergence to the center. We conclude by a few heuristic remarks.
 \section{A global existence result}  For convenience we recall our basic equation \eqref{diss}
\begin{equation}\label{dis} u'' + \delta u' + c  \frac {u} {|u|^3} = 0 \end{equation}  where  we simply wrote $c$ instead of $c_0$, since for mathematicians $c$ is not always the velocity of light. In order to solve this equation with $u = u(t) \in \re^2 = \mathbb{C}$ we introduce the amplitude and the phase 
$$ u(t) = r(t) e^{i\theta(t)} $$ and we shall drop the variable $t$ when it does not lead to confusion. Here all derivatives are with respect to $t$. From the formulas 
$$ u' = r' e^{i\theta} + ir \theta'  e^{i\theta}= ( r' + ir \theta' )e^{i\theta} =  ( r' + ir \theta' ) \frac{u} {|u|} $$ 
$$ u''= [(r'' -r \theta'^2) + i( 2r'  \theta'  + r \theta'' )] \frac{u} {|u|} $$ we conclude that \eqref{dis} is equivalent to the system of two real equations 
\begin{equation}\label{r} r'' -r \theta'^2+ \delta r' +  \frac {c} {r^2} = 0 \end{equation} 
\begin{equation}\label{theta} 2r'  \theta'  + r \theta'' + \delta r\theta' = 0 \end{equation} 
We observe first that when trying to solve \eqref{dis} for $t\ge 0$, the initial value $u(0) = 0$ is excluded by the singularity, while whenever $u(0)\not= 0$ we shall obtain  at least a local solution for any initial velocity $u'(0)$. Moreover, multiplying \eqref{theta} by $r$ we reduce it to  
 $$ ( r^2 \theta') '   + \delta r^2\theta' = 0 $$ so that \begin{equation}\label{theta1}\forall t\ge 0, \quad r^2(t) \theta' (t) = M e^{-\delta t}\end{equation}  with $$ M =:  r^2(0) \theta' (0) $$  This relation expresses the variation of the momentum of the solution, which was constant in the conservative case $\delta =0$ and decays exponentially when $\delta>0$. By plugging in \eqref{r} the value of $\theta'$ given by \eqref{theta1}, we are left with an equation involving $r(t) $ only:
  \begin{equation}\label{r1} r'' -\frac{M^2e^{-2\delta t}}{r^3} +  \frac {c} {r^2} + \delta r' = 0 \end{equation} 
We now state our first result \begin{thm} Assuming $ M \not=0 $, the unique local non-vanishing solution $r= r(t)$ of \eqref{r1} with any initial data $r(0)= r_0\not=0; r'(0) = r'_0 \in \mathbb{R} $   is global and we have for some constants $C>0, \eta>0 $ 
$$  \forall t\ge 0, \quad | r' (t)| \le C e^{\delta t}$$ 
$$  \forall t\ge 0, \quad r (t)  \ge \eta  e^{-2\delta t}$$  \end{thm} 
\begin{proof} The existence and uniqueness of local non-vanishing solutions is obvious. Let us introduce the scalar function $F$, defined as long as the solution $r$ exists and does not vanish  by the formula:
 \begin{equation}\label{F} F(t) : = \frac{r'^2(t)}{2}  - \frac {c} {r(t)} + \frac{M^2}{2} \frac{e^{-2\delta t}}{ r^2(t)}  \end{equation} We have immediately
   $$ F'(t) = r' (t) ( r''(t) + \frac {c} {r^2(t)} -  \frac{M^2  e^{-2\delta t}}{ r^3(t)}  ) -  \frac{M^2 \delta e^{-2\delta t}}{ r^2(t)}  =  -\delta r'^2(t) -  \frac{M^2 \delta e^{-2\delta t}}{ r^2(t)} < 0 $$ 
    In particular for all $ t\in [0, Tmax)$ we have 
    $$  \frac{r'^2(t)}{2}  - \frac {c} {r(t)} + \frac{M^2}{2} \frac{e^{-2\delta t}}{ r^2(t)} \le F_0: = F(0) $$ which can be rearranged in the more convenient form 
    
    \begin{equation} \frac{r'^2(t)}{2}  + \frac{M^2}{2} \left(\frac{e^{-\delta t}}{ r(t)}- \frac{c}{M^2} e^{\delta t}\right)^2 \le F_0 + \frac{c^2}{2M^2} e^{2\delta t}  \end{equation} providing  at once the two inequalities 
    $$ \forall t\in [0, Tmax), \quad r'^2(t) \le 2(F_0 + \frac{c^2}{2M^2} e^{2\delta t}) \le (2 F_0 + \frac{c^2}{M^2} ) e^{2\delta t} $$ and 
    $$ \forall t\in [0, Tmax), \quad \frac{e^{-\delta t}}{ r(t)} \le \frac{c}{M^2} e^{\delta t} + \sqrt{\frac{2}{M^2}F_0 +\frac{ c^2 }{M^4} e^{2\delta t}  } \le Ke^{\delta t} $$ The conclusions follow immediately with $$ C = (2 F_0 + \frac{c^2}{M^2} )^{1/2}; \quad \eta = \frac{1}{K}. $$
 \end{proof}
 
  \begin{cor} Assume  \begin{equation} Im ( \overline {u_0 } u'_0 ) \not = 0 \end{equation} Then the local solution of \eqref{dis} with initial conditions $ u(0) = u_0 ; u'(0) = u'_0$  is global and satisfies for some constants $D>0, \eta>0 $ 
\begin{equation} \forall t\ge 0, \quad | u' (t)| \le D e^{\delta t}\end{equation}
\begin{equation}  \forall t\ge 0, \quad |u (t)| \ge \eta  e^{-2\delta t}\end{equation}  
  
 \end{cor} 
\begin{proof} I suffices to prove that the two above inequalities are satisfied for all $ t\in [0, Tmax)$, the maximal existence time for a non-vanishing solution $u$ of \eqref{dis} with initial conditions $ u(0) = u_0 ; u'(0) = u'_0$.  Selecting for $\theta$ a continuous determination of the argument of $u(t)$, the pair $(r, \theta) $ satisfies the equivalent system of equations \eqref{r}-\eqref{theta}. In addition an immediate calculation shows that $\theta' =  Im ( \frac{\overline{u}} {|u|} u' )$  and in particular $$ Im ( \overline {u_0 } u'_0 ) = r(0) \theta'(0)$$ By the previous theorem, the solution $r(t)$ of the scalar equation for the radius is global and satisfies the concluding inequalities. Then the pair $(r, \theta) $ correspond to the non-vanishing solution $u$ of \eqref{dis} with initial conditions $ u(0) = u_0 ; u'(0) = u'_0$ which is therefore global with $ |u (t)| \ge \eta  e^{-2\delta t}$.  For the last inequality we observe that $$ |\theta'(t) r(t) \le Me^{-\delta t} r^{-1}(t) \le \frac{M}{\eta} e^{\delta t}  $$ Then $$ |u'(t)| = | r' + ir \theta'| \le | r'(t)| + | r (t)\theta'(t)| \le (C+  \frac{M}{\eta} ) e^{\delta t}  : = D  e^{\delta t} .$$\end{proof}

 \section{Bounded trajectories}  The next result does not require the condition $M\not=0 $. \begin{thm}  Let  $u_0\not=0 $ and assume  the initial smallness condition  
 \begin{equation}\label{small} |u_0||u'_0|^2 < 2c \end{equation}  Then the local solution $u$  of \eqref{dis} on $[0, T)$ with initial conditions $ u(0) = u_0 ; u'(0) = u'_0$   satisfies the inequality 
 \begin{equation}\label{bound} \forall t\in [0, T) , \quad |u(t)| \le \frac{2c|u_0|}{2c - |u_0||u'_0|^2}\end{equation} In particular if $u$ does not vanish in finite time, $u$ is a global bounded non-vanishing solution .  
  \end{thm} 
  
  \begin{proof} We start from the inequality  $$  \frac{r'^2(t)}{2}  - \frac {c} {r(t)} + \frac{M^2}{2} \frac{e^{-2\delta t}}{ r^2(t)} \le F_0=   \frac{r'^2_0}{2}  - \frac {c} {r_0} + \frac{M^2}{2 r^2_0}  $$  and we observe that if $F_0<0$, the inequality   $ \displaystyle \frac {c} {r(t)} \ge - F_0 = |F_0| $  implies $ \displaystyle r(t) \le  \frac {c} {|F_0|}. $ Now we compute $F_0$ in terms of the initial data.  In fact we have the formula $$ |u'|^2 = r'^2 + \theta'^2 r^2 = r'^2+ \frac{M^2}{ r^2} $$ so that for $t=0$ we find 
  
  $$  F_0=   \frac{|u'_0|^2}{2}  - \frac {c} {r_0} = \frac{|u'_0|^2}{2} - \frac {c} {|u_0|} $$ therefore the condition $F_0<0$ is equivalent to $ |u_0||u'_0|^2 < 2c$.  The previous observation now gives that under condition \eqref {small}, we have 
  $$ r(t) \le  \frac {c} {\displaystyle  \frac {c} {|u_0|}- \frac {|u'_0|^2}{2} } =  \frac{2c|u_0|}{2c - |u_0||u'_0|^2} $$  \end{proof}
  
   \section{Convergence to 0 of non-vanishing bounded solutions } As in the case of the conservative equation \eqref{cons} which is well known to have elliptic and parabolic trajectories,  we suspect that \eqref{dis} may have some unbounded solutions. But the next result shows that in sharp constrast with  \eqref{cons}, \eqref{dis} does not have any periodic trajectory at all. \begin{thm} For any solution $u$ of \eqref{dis}  such as $|u(t)|$ is global, positive and bounded on $\re^+$, we have 
   
 \begin{equation}\label{conv0} \lim_ {t\to +\infty} |u(t)| = 0\end{equation}   
  \end{thm} 
  
  \begin{proof} We introduce the total energy  $$  E(t) : =   \frac{1}{2} |u'|^2(t )- \frac {c} {|u(t)|} $$  Since $u$ never vanishes, it is clear that $u\in C^2(\re^+)$ and we have\begin{equation} \label{E'} E'(t) = -\delta |u'(t)|^2\end{equation}   In particular, $E(t)$ is non-increasing. Then we have two possibilities \medskip
  
 { \bf Case 1}  \begin{equation}\label{1} \lim_ {t\to +\infty} E(t)  =  -\infty \end{equation} Then since $\displaystyle \frac {c} {|u(t)|}\ge -E(t) $ we conclude that 
 $$ \lim_ {t\to +\infty} |u(t)| = 0$$.\medskip

 { \bf Case 2}  \begin{equation}\label{2} \lim_ {t\to +\infty} E(t)  = E^* > -\infty \end{equation} Then  $$ \forall t\in \re^+, \quad E(0)- E(t) = \delta\int _0^{t} |u'(s|^2ds  $$  In particular $u'\in L^2(\re^+, \mathbb{C}).$  We have assumed that $u(t) $ is bounded, hence  precompact in $\re^+$ with values in $\mathbb{C}$. Therefore if \eqref{conv0} is not satisfied we may assume that for some sequence $t_n$ tending to $+\infty$  \begin{equation}\label{convw} \lim_ {n\to +\infty} u(t_n) = w \not= 0\end{equation}  On the other hand we have 
  \begin{equation}\label{conv'} \lim_ {n\to +\infty} u'(t_n+s) =  0\end{equation} in the strong topology of $L^2(0, 1)$ and in particular, by Cauchy-Schwarz-inequality we find  \begin{equation}\label{convw+} \lim_ {n\to +\infty} u(t_n+s) = w \end{equation} uniformly on $[0, 1]$. Since $w \not= 0$ this implies 
  \begin{equation}\label{convf} \lim_ {n\to +\infty} c \frac{u(t_n+s)}{|u(t_n+s)|^2} = c \frac{w}{|w|^2 } \end{equation}  uniformly on $ [0, 1]. $  But then by the equation  \begin{equation}\label{conv''} \lim_ {n\to +\infty} u''(t_n+s) =  c \frac{w}{|w|^2 } \end{equation} in the strong topology of $L^2(0, 1)$. This is contradictory with \eqref {conv'}. Indeed it implies for instance 
  $$  \lim_ {n\to +\infty} \int_0^1 s(1-s) u''(t_n+s) ds = c \int_0^1 s(1-s)  ds \frac{w}{|w|^2} = z\not=0 $$ while on the other hand 
   $$   \int_0^1 s(1-s) u''(t_n+s) ds =  -   \int_0^1 (1-2s)  u'(t_n+s) ds \longrightarrow  0 $$ This contradiction concludes the proof. 
  \end{proof} 
  
  As an immediate consequence of the previous theorem and the results of Sections 2 and 3, we obtain
  
   \begin{cor} Assume  $ Im ( \overline {u_0 } u'_0 ) \not = 0  $ and  $|u_0||u'_0|^2 < 2c$. Then the local solution of \eqref{dis} with initial conditions $ u(0) = u_0 ; u'(0) = u'_0$  tends to $0$ as $t$ tends to infinity.   
 \end{cor}

 \section{A formal calculation for spiraling solutions} We now  try  a formal calculation suggesting that  solutions with norm equivalent to a decreasing exponential for $t$ large might however  exist.  We would like to see which sort of solutions would replace the circular trajectories of the conservative problem \eqref{cons} (cf. also section 2.) Observing that such trajectories correspond to a constant angular velocity and for them the nonlinear term of the radial equation vanishes, it is natural to try 
 $$ r(t) = \frac{M^2}{c} e^{-2\delta t}. $$  By the argument above his cannot be an exact solution of the equation and indeed  $$ r''+ \delta r' = \frac{M^2}{c} 2\delta^2  e^{-2\delta t} \not = 0 $$ but at least the result is non-singular and even becomes very small if $\delta$ is small. This gives the idea to look for a solution of the form \begin{equation}  r(t) = \frac{M^2}{c} e^{-2\delta t} + \varepsilon (t)\end{equation}  where $\varepsilon (t)$ would be negligible with respect to $\frac{M^2}{c} e^{-2\delta t}$, either uniformly, or at least for $t$ large.  Plugging this in the equation we find $$ c\varepsilon(t) = -r^3(r''+\delta r')  $$ If now we replace both 
 $$r^3\sim  \frac{M^6}{c^3} e^{-6\delta t}; \quad r''+\delta r'\sim \frac{M^2}{c} (2\delta^2) e^{-2t} $$ by their alleged ``main part" we find 
 
 $$ \varepsilon(t) \sim  - 2 \delta^2 \frac{M^8}{c^5}  e^{-8\delta t}  $$ which is indeed quite negligible with respect to $\frac{M^2}{c} e^{-2\delta t}$ for $t$ large and even uniformly if $\delta$ and  $\frac{M^6}{c^4}$ are both small enough. To see the rule we main consider trying the next step. the calculations are slightly more involved but we get rather easily 
\begin{equation} \varepsilon(t) \sim  - 2 \delta^2 \frac{M^8}{c^5}  e^{-8\delta t} + 124 \delta^4 \frac{M^{14}}{c^9}e^{-14\delta t} ...\end{equation}
 From the algebraic point of view, it is not difficult to see that the process can be continued indefinitely with the appearance of a single new exponential term at each step. So that we can dream of a solution of the form 
 \begin{equation} r(t) \sim \frac{M^2}{c} \sum _{n=0}^{\infty} (-1)^n C_n \delta^ {2n } \left(\frac{M^6}{c^4}\right)^ne^{-(2+6n)\delta t}\end{equation}
 The only difficulty for convergence would be to control the growth of $C_n $.  Actually $C_n $ seems to grow like a factorial, precluding any sort of convergence for any fixed value of $t$. Then in which sense would the above calculations approach a solution? At each step  we have an approximate solution with a very close accuracy for $t$ large. For the time being the moral of the story remains a mystery for the author.  In any case  it is clear that a different method is needed to find the solutions we are looking for.
  
  \section{Concluding remarks} The previous calculations suggest the possible existence of special orbits with an asymptotically (exponentially fast) spiraling convergence to the center. More precisely if we succeed in finding, by whichever method, a solution $r$ of \eqref {r1}with  $$ r(t) \sim Ce^{-2\delta t} $$ for a certain $C= C(M)>0$, from this we can build a solution $u= re^{i\theta}$ with $$ \theta'^2\sim K e^{6\delta t }$$ so that $ |u'|^2 = r'^2 + \theta'^2 r^2 \ge k e^{2\delta t } $ for some positive $k$ and $t$ large. The kinetic energy blows-up exponentially in such a case, which is not absurd since the potential energy also. \begin{rmk} This construction seems interesting, but even in case it works  this is not entirely satisfactory for the following reason: as in the case of \eqref{tired}, the family of spiraling orbits built in this way would depend on only two real parameters (the moment and the initial phase) whereas the phase space has four dimensions. \end{rmk} 
 \begin{rmk} For the moment, despite the formal calculation of the previous section,  we know nothing about the rate of decay of any solution to $0$ . The only certain information that we have on asymptotics is the inequality  $ r (t)  \ge \eta  e^{-2\delta t}$ for some $\eta>0$. This inequality can be refined asymptotically, relying on  the inequality  
  $$ \forall t\in [0, Tmax), \quad \frac{e^{-\delta t}}{ r(t)} \le \frac{c}{M^2} e^{\delta t} + \sqrt{\frac{2}{M^2}F_0 +\frac{ c^2 }{M^4} e^{2\delta t}  }  $$ 
 to yield $$ \liminf_{t\to\infty}  r (t)  e^{2\delta t} \ge \frac{M^2}{2c}.$$ \end{rmk} 
  \begin{rmk} In the case  $M=0$, it should be possible to study if, as in the conservative case, solutions with  initial velocity directed towards the center vanish in finite time, if necessary under some additional conditions. We did not study this question here since it is not our main concern.  \end{rmk} 
  
  \begin{rmk} After the simple model of Rutherford-Bohr explained in \cite{Bohr, Rutherford} and after the introduction of undulatory mechanics by L. De Broglie, the purely probabilistic model of E. Schrodinger (cf. \cite{Schrodinger}) has  been considered the best atomic model since 1926. It does not  seem obvious at all to introduce a damping mechanism in that model, and it might very well happen that in order to do that, an effective coupling between deterministic corpuscular and probabilistic wave conceptions of particles has to be devised, as was always advocated by L. De Broglie and even suggested by  E. Schrodinger himself in \cite{Schrodinger}.  \end{rmk}

\end{document}